\RequirePackage{fix-cm}
\documentclass[smallextended]{svjour3}       

\smartqed  
\usepackage{graphicx}
\usepackage{amssymb}
\usepackage{enumitem}
\usepackage{amsmath}				
\DeclareMathOperator{\eco}{e-conv}
\DeclareMathOperator{\Keco}{\emph{K}-e-conv}						

\DeclareMathOperator{\conv}{conv}
\DeclareMathOperator{\Kconv}{\emph{K}-conv}
\DeclareMathOperator{\Kcl}{\emph{K}-cl}

\DeclareMathOperator{\cl}{cl}

\DeclareMathOperator{\dom}{dom}
		
\DeclareMathOperator{\epi}{epi}

\DeclareMathOperator{\gph}{gph}

\def\supp{\mathop{\rm sup}}

\newcommand{\en}{\rightarrow}											
\newcommand{\R}{\mathbb{R}}												
\newcommand{\Ramp}{\overline{\R}}					

\newcommand{\ci}{\left\langle}										
\newcommand{\cd}{\right\rangle}										

\begin{document}

\title {Set-valued evenly convex functions: characterizations and c-conjugacy}

\author{M.D. Fajardo}


\institute{M.D. Fajardo \at
             Mathematics Department, Faculty of Sciences, University of Alicante, Alicante, Spain\\
              Tel.: +34-965903532\\
              Fax: +34-965903531\\
              \email{md.fajardo@ua.es} }

\maketitle

\begin{abstract}
In this work we deal with set-valued functions with values in the power set of a separated locally convex space where a nontrivial pointed convex cone induces a partial order relation. A set-valued function is  evenly convex if its epigraph is an evenly convex set, i.e., it is the intersection of an arbitrary family of open half-spaces. In this paper we characterize evenly convex set-valued functions as the pointwise supremum of its set-valued e-affine minorants. Moreover, a suitable conjugation pattern will be developed for these functions, as well as the counterpart of the biconjugation Fenchel-Moreau theorem.
\end{abstract}

\section{Introduction}
\label{intro}
\emph{Evenly convex} sets (e-convex in brief)  were introduced by Fenchel \cite{F1952} in $\R^n$  as those sets which can be expressed as the intersection of an arbitrary (possibly empty) family of open half-spaces. He initiated the study of this kind of sets defining a natural polarity operation and mentioning some of their properties. Later, they appeared in some works related to quasiconvex programming (see \cite{Ml1983}, \cite{Ml1988}, \cite{PP1984} and \cite{PV1990}) and mathematical economics \cite{ML1991}. E-convex sets are the solution sets of linear systems of either weak or strict inequalities, and from this point of view, their properties have been studied in \cite{GJR2003} and \cite{GR2006}, whereas some important characterizations of e-convex sets in terms of their sections and projections were given in \cite{KMZ2007}. Finally, some properties of this class of sets in locally convex spaces can be found in \cite{DML2002}.

 In a natural way, the concept of even convexity was applied to an extended real-valued function defined on $\R^n$ in \cite{RVP2011}, where a function is said to be e-convex if its epigraph is e-convex. In that work, the main properties of these functions, which extend the class of convex and lower semicontinuous functions, were studied. In particular, some characterizations were provided. Later, in  \cite{MLVP2011}, the notion of e-convex function was considered for extended real-valued functions defined on locally convex spaces, and proper e-convex functions were characterized as the pointwise supremum of their e-affine minorants. This characterization allowed, also in \cite{MLVP2011}, the definition of a suitable conjugation scheme for e-convex functions. This pattern  is called \emph{$c$-conjugation}, and it was inspired by a survey written  by Mart\'inez-Legaz \cite{ML2005} where generalized convex duality theory is applied to quasiconvex programming. As it is shown in \cite{MLVP2011}, the $c$-conjugation scheme allowed to set forth the e-convex counterpart of the well-known Fenchel-Moreau theorem. Recently, some works have been developed around duality theory in evenly convex optimization (\cite{F2015}, \cite{FVR2016}, \cite{FV2016SSD}, \cite{FV2017}, \cite{FV2018}), building different c-conjugate dual problems by means of the perturbation approach and obtaining regularity conditions expressed in terms of the even convexity of the perturbation function. A monograph  presenting the state-of-art on even convexity in finite dimensional spaces has been published very recently (see \cite{FGRV2020}).
 
 In this paper we extend the notion of even convexity to set-valued functions which map a separated locally convex space $X$  into the power set $\mathcal{P}(Z)$ of another separated locally convex space $Z$, assuming that there exists a partial order in $Z$ induced by a pointed convex cone $K \subset Z$. Every concept and result from the even convexity scalar theory can be translated in the field of set-valued functions by choosing appropiate definitions of properness, e-affine set-valued minorants, set-valued c-conjugate function, etc, and, moreover, extending operations, like the meaning of pointwise supremum or infimum. We will assume the same framework as in \cite{H2009}, where a characterization for proper closed convex set-valued functions by its continuous affine minorants and an appropiate Fenchel conjugate are provided. Also in that paper introduction section, different approaches which have been developed for conjugation patterns suitable for closed convex set-valued functions are highlighted.

The layout of this work is as follows. Section \ref{sec2} contains preliminary results on set-valued functions, as well as fundamental definitions and results on e-convex sets and extended real-valued functions to make the paper self-contained.
Section \ref{sec3} is dedicated to characterizations of set-valued proper and e-convex functions in terms of different kinds of affine  minorants.
Section \ref{sec4} is devoted to a suitable conjugation pattern for set-valued e-convex functions, which will provide a biconjugation theorem.  

\section{Preliminaries}
\label{sec2}
Let $X$ and $Z$ be nontrivial separated locally convex spaces, we denote by $X^*$ and $Z^*$ their respective topological dual spaces, and by $\langle \cdot , \cdot \rangle$ the duality product in both cases: $\langle x, x^* \rangle = x^* (x)$ for all $x \in X, x^* \in X^*$ (analogously in $Z$ and $Z^*$). In this paper we will deal with set-valued functions $f:X \to \mathcal{P}(Z)$, where $\mathcal{P}(Z)$ is  the set of all the subsets of $Z$, including the empty set.  The algebraic structure of $\mathcal{P}(Z)$ is determined by the Minkowski sum $A+B=\{ a+b:a\in A, b\in B\}$,  for all $A,B \in\mathcal{P}(Z)$, and the product with a nonnegative real number  $tA=\{ta:a\in A \}$. We set $A+\emptyset =\emptyset+A=\emptyset$, for all $A \in\mathcal{P}(Z)$, $t\emptyset=\emptyset$, for all $t \in \R \setminus \{0\}$ and $0 \emptyset = \{0\}$. We write $A+ \{z\}=A+z$, $A+\{-z\}=A-z$. We set, for $A \neq \emptyset$ or $A \neq Z$, $(-1)A=-A=\lbrace -a: a \in A\rbrace$, whereas $-Z= \emptyset$ and $-\emptyset=Z$. Finally, we write $A+(-B)=A-B$.\\

In \cite{KTZ2015} several order relations which are used in order to formulate corresponding solution concepts for a set-valued optimization problem are introduced, allowing more practical comparisons among values of the set-valued objective function. Let $K \subset Z$  be a convex cone with $\{0\} \varsubsetneq K \varsubsetneq Z$.  Then $K$ induces in $Z$ a quasiorder (reflexive and transitive binary relation) in the following way: $x \leq_K y$ if and only if $y-x \in K$, for all $x,y \in Z$. Moreover, $K$ also induces a quasiorder relation in  $\mathcal{P}(Z)$ as $A \leq_K B$ if and only if $B \subseteq A+K$, see \cite [Definition 2.1] {KTX1997}, called  \emph{lower set less order relation} in \cite[Definition 2.6.9]{KTZ2015}. If we consider the subset of  $\mathcal{P}(Z)$
$$\mathcal{P}^{a}_K(Z):=\lbrace {A \in \mathcal{P}(Z) : A=A+K } \rbrace,$$
we will have that $\leq_K$ induces a partial order (reflexive, antisymmetric and transitive binary relation) in $\mathcal{P}^{a}_K(Z)$, in fact, for $A, B \in \mathcal{P}^{a}_K(Z)$, $A \leq_K B$ if and only if $A \supseteq B$. With the Minkowski sum and a little modification in the scalar product, $0A=K$, for all $A \in \mathcal{P}^{a}_K(Z)$, we will have that  $ \mathcal{P}^{a}_K(Z)$ is a partially ordered conlinear space (see \cite{H2005} for the definition and structural properties of conlinear spaces, also \cite[Appendix]{H2009} for a short introduction). Moreover, $(\mathcal{P}^{a}_K(Z), \leq_K)$ is a complete lattice, where, for $ \mathcal{A} \subseteq \mathcal{P}^{a}_K(Z)$, $\sup\lbrace \mathcal{A}, \leq_K \rbrace = \cap_{A \in \mathcal{A}} A$ and $\inf\lbrace \mathcal{A}, \leq_K \rbrace = \cup_{A \in \mathcal{A}} A$. In the case $ \mathcal{A}$ is empty, we put $\sup\lbrace \mathcal{A}, \leq_K \rbrace = Z$ and $\inf\lbrace \mathcal{A}, \leq_K \rbrace = \emptyset$. The following definitions are well-known.\\

A set-valued function $f:X \to \mathcal{P}(Z)$ is said to be \emph{proper} if its effective domain
\begin{equation*}
\dom f:=\{x \in X : f(x) \neq \emptyset \}
\end{equation*}
is nonempty and $f(x) \neq Z$, for all $x \in \dom f$. If $f$ is not proper, we say that it is \emph{improper}. It is said to be \emph{convex} (\emph{closed}, resp.) if its graph
$$\gph f:=\lbrace { (x,z) \in X \times Z : z \in f(x) } \rbrace$$
is a convex (closed with respect to the product topology, resp.) set in $X \times Z$.\\

We associate to any set-valued function $f:X \to \mathcal{P}(Z)$  the set-valued function 
\begin{equation} \label{def_fK}
f_K:X \to \mathcal{P}^{a}_K(Z), f_K(x):=f(x)+K.
\end{equation}
\begin{definition}
Let $f:X \to \mathcal{P}(Z)$ be a set-valued function. A set-valued minorant of $f$ is $g:X \to \mathcal{P}^a_K(Z)$  verifying that, for all $x \in X$, $g(x) \leq_K f(x)$, i.e., $f_K
(x) \subseteq g_K(x)$.
\end{definition}
\begin{definition} \cite[page 26] {KTZ2015}\label{def1}
A set-valued function $f:X \to \mathcal{P}(Z)$ is said to be \emph{K-convex} if the set-valued function 
$f_K$ defined in (\ref{def_fK}) is convex.
\end{definition}
Therefore, $K$-convexity means that the \emph {K-epigragh} of $f$,
\begin{equation*}
\epi_K f:=\{(x,z) \in X \times Z : z \in f(x)+K\}=\gph f_K \subset X \times Z
\end{equation*}
is  a convex set, or, equivalently, for all $t \in (0,1)$, and $x_1, x_2 \in X$
\begin{equation*}
f \left(tx_1+(1-t)x_2 \right)\leq_K tf(x_1)+(1-t)f(x_2).
\end{equation*} 
It is easy to see that, if  $f:X \to \mathcal{P}(Z)$ is $K$-convex, $f_K(x)$ is a convex set in $Z$, for all $x \in X$, and we can reduce the images space of $f_K$ to
$$\mathcal{Q}^{a}_K(Z):=\lbrace {A \in \mathcal{P}(Z) : A=\conv (A+K) } \rbrace,$$
which is a partially ordered conlinear subspace of $\mathcal{P}^{a}_K(Z)$. Hence, for a set-valued function $f:X \to \mathcal{P}(Z)$, its \emph{K-convex hull} 
$$ \Kconv f: X \to \mathcal{Q}^{a}_K(Z),$$
defined by $z \in (\Kconv f)(x)$, for $x \in X$ and $z \in Z$  if and only if $(x,z) \in \conv (\epi_K f)$, i.e.,
$$\epi_K(\Kconv f)=\gph (\Kconv f)=\conv (\epi_K f),$$
is (unique) well-defined. Moreover,  $\Kconv f$ is the largest $K$-convex minorant of $f$, since denoting by $\mathcal{G}= \lbrace{ g : g \text{ is a $K$-convex minorant of } f \rbrace}$, then
$$ (\Kconv f)(x)=\sup { \lbrace { g_K(x): g \in \mathcal G} \rbrace}. $$

In the following closed set-valued function definition (\cite[Definition 3]{H2009}), like in Definition \ref{def1}, we have also specified that the closedness of the function is related to the convex cone $K$, for the sake of consistency.
\begin{definition}
A set-valued function $f:X \to \mathcal{P}(Z)$ is said to be \emph{K-closed}  if the set-valued function $f_K$ is closed, i.e., $\epi_K f$ is a closed set with respect to the product topology on $X \times Z$. 
\end{definition}

It  is not difficult to prove that, if  $f:X \to \mathcal{P}(Z)$ is $K$-closed, $f_K(x)$ is a closed set in $Z$, for all $x \in X$, and we can reduce the images space of $f_K$ to
$$\mathcal{P}^{t}_K(Z):=\lbrace {A \in \mathcal{P}(Z) : A=\cl (A+K) } \rbrace,$$
which is a partially ordered conlinear space if we consider the modified Minkowski sum
$$A \oplus B:=\cl(A+B),$$
and the modified scalar product $0A:=\cl K$.

Hence, for a set-valued function $f:X \to \mathcal{P}(Z)$, its \emph{K-closed hull} 
$$ \Kcl f: X \to \mathcal{P}^{t}_K(Z),$$
is defined by $z \in (\Kcl f)(x)$, for $x \in X$ and $z \in Z$  if and only if $(x,z) \in \cl (\epi_K f)$, i.e.,
$$\epi_K (\Kcl f)=\gph (\Kcl f)=\cl (\epi_K f).$$

Clearly, the $K$-closed $K$-convex hull of a set-valued function $f:X \to \mathcal{P}(Z)$, $$ \Kcl\conv f: X \to \mathcal{Q}^{t}_K(Z),$$
is defined by $z \in (\Kcl \conv f)(x)$, for $x \in X$ and $z \in Z$  if and only if $(x,z) \in \cl \conv (\epi_K f)$, i.e.,
$$\epi_K (\Kcl \conv f)=\gph (\Kcl \conv f)=\cl \conv (\epi_K f),$$
where $$\mathcal{Q}^{t}_K(Z):=\lbrace {A \in \mathcal{P}(Z) : A=\cl \conv(A+K) } \rbrace,$$
which is a partially ordered conlinear subspace of $\mathcal{P}^{t}_K(Z)$.

We will use the following standard notation (see, for instance, \cite{Z2002}) for open half-spaces in locally convex spaces, either in $X$  or $Z$,  for $y^* \in X^*$  ($z^* \in Z^*$) and $\alpha   \in \R$  ($\beta \in \R$): 
\begin{equation*}
H^-_{y^*,\alpha}:=\{x \in X: \langle x, y^* \rangle < \alpha \},\\
H^-_{z^*,\beta}:=\{z \in Z: \langle z, z^* \rangle < \beta \}.
\end{equation*}

In what follows, we denote by $K^*$ the negative polar cone of $K$, i.e.,
$$K^*= \lbrace z^* \in Z^* : \langle z,z^* \rangle \leq 0, \text{ for all } z \in K \rbrace,$$
and, for $(x^*,z^*) \in X^*\times( K^*\setminus \{0\})$, we will use the following set-valued functions, $S_{(x^*,z^*)}, \bar {S}_{(x^*,z^*)}:X \to  \mathcal{Q}^a_K(Z)$,  with open and closed half-spaces in $Z$ as images, respectively, 
\begin{equation}\label{defS}
\begin{array} {lcl}
S_{(x^*,z^*)}(x)&:=\{z \in Z: \langle x, x^* \rangle + \langle z, z^* \rangle < 0\},\\
\bar S_{(x^*,z^*)}(x)&:=\{z \in Z: \langle x, x^* \rangle + \langle z, z^* \rangle \leq 0\}.
\end{array}
\end{equation}

As it was said in the introduction, Fenchel defined an e-convex set  $C\subseteq \R^n$ as a set which can be expressed as the intersection of an arbitrary family of open half-spaces. In \cite{DML2002} this kind of sets are defined in  locally convex spaces in the following equivalent way, which will be very useful in the sequel.
\begin{proposition}\label{defeconv}
A set $C\subseteq X$ is  e-convex  if for every point $x_0\notin C$, there exists $x^*\in X^*$ such that $\ci x-x_0,x^*\cd<0$, for all $x\in C$. 
\end{proposition}
 
For a set $C\subseteq X$, the \emph{e-convex hull} of $C$, $\eco C$, is the smallest e-convex set in $X$ containing $C$. For a convex subset  $C\subseteq X$, it always holds $C\subseteq \eco C \subseteq \cl C$.
This operator is well defined because the class of e-convex sets is closed under arbitrary intersections.
Since $X$ is a separated locally convex space, $X^*\neq \{0\}$. As a consequence of the Hahn-Banach theorem, it also holds that $X$ is e-convex and every closed or open convex set is e-convex as well.\\
\begin{definition} \cite{RVP2011}
A function $f:X \to \Ramp:=\R\cup \{\mp \infty\}$ is \emph{e-convex}  if its epigraph $\epi f=\lbrace (x,\alpha): f(x) \leq \alpha \rbrace$  is e-convex in $X\times \R$.
\end{definition}

Clearly, any lower semicontinuous convex function  $f:X \to \Ramp$  is e-convex, but the converse does not hold, as the following example shows.
\begin{example}
Let $f:\R \to \Ramp$,
\begin{equation*}
f(x) :=\left\{
\begin{array}{ll}
x^2& \text{if } x<0, \\
+\infty & \text{otherwise.}
\end{array}
\right.
\end{equation*}
Its epigraph is an e-convex set, since we can find, for any point not belonging to $\epi f$,  a hyperplane passing through the point but with empty intersection with $\epi f$. Nevertheless, the convex set $\epi f$ is not closed.
\end{example}

The e-convex hull of a function $f:X \to \Ramp$, $\eco f$, is defined as the largest e-convex minorant of $f$. It holds that, if $\bar f$ is the lower semicontinuous convex hull of any function $f:X \to \Ramp$, then $\bar f(x) \leq \eco f(x) \leq f(x)$, for all $x \in X$.

Now we extend the concept of even convexity to set-valued functions. Like in convexity and closedness definitions, we will give a general definition and another one depending on the chosen convex cone $K$.
\begin{definition}
 A set-valued function $f:X \to \mathcal{P}(Z)$ is \emph{e-convex}  if its graph is an e-convex set in $X \times Z$.
\end{definition}
\begin{definition}
 A set-valued function $f:X \to \mathcal{P}(Z)$ is \emph{K-e-convex}  if $f_K$ is an e-convex function. It is equivalent to saying that its $K$-epigraph is an e-convex set in $X \times Z$. 
\end{definition}

\begin{remark}\label{rem1}
The definition of $K$-e-convex set-valued  function collapses into the e-convex function definition  $f: X \to \Ramp$ if we consider a set-valued function $f^s:X \to \mathcal {P}(\R)$, with $K=\R_+$, which can be associated to $f$, having the same epigraph, i.e., $\epi f=\epi_K f^s$. It is enough to take $f^s(x)=\{f(x)\}$, when $f(x) \in \R$, $f^s(x)=\emptyset$ if $f(x)=+\infty$ and $f^s(x)=\R$ if $f(x)=-\infty$. It holds that $f$ is e-convex if and only if $f^s$ is $K$-e-convex. Indeed, we have
$$\epi_K f^s=\lbrace { (x,z):z\in f^s(x)+K \rbrace}=\lbrace { (x,z): f(x) \leq z \rbrace}=\epi f.$$
\end{remark}

The class of $K$-e-convex functions contains the class of $K$-closed $K$-convex functions, but this inclusion is strict.

\begin{example}
Let $f:\R \to \mathcal{P}(\R)$, with $K=\R_{+}$,
\begin{equation*}
f(x) :=\left\{
\begin{array}{ll}
[x^2, 2x^2] & \text{if } x>0, \\
\emptyset & \text{otherwise.}
\end{array}
\right.
\end{equation*}
We have that $$\epi_K f= \lbrace {(x,z):x>0, z \geq x^2 \rbrace},$$ 
which is an e-convex nonclosed set.
\end{example}

\begin{proposition}
 If  $f:X \to \mathcal{P}(Z)$ is $K$-e-convex, then $f_K(x)$ is an e-convex set in $Z$ for all $x \in X$.
\end{proposition}
\begin{proof}
Let us show that, for $\bar x \in \dom f$, $f(\bar x)+K$ is e-convex (it is evident when $\bar x \not \in \dom f$  and $f(\bar x)+K=\emptyset$). \\
Take a point $\bar z \notin f(\bar x)+K$. Then $(\bar x, \bar z) \notin \epi_K f$ which is e-convex, and according to Proposition \ref {defeconv}, there exists $0 \neq (x^*, z^*) \in X^* \times Z^*$ such that 
$$ \langle x-\bar x, x^* \rangle + \langle z-\bar z, z^* \rangle < 0,$$
for all $(x,z) \in \epi_K f$. Now, for any point $z \in f(\bar x)+K$, since  $(\bar x, z) \in \epi_K f$, we have
$$ \langle z-\bar z, z^* \rangle < 0,$$
and it shows, again in virtue of Proposition \ref {defeconv}, that $f(\bar x)+K$ is e-convex.\qed
\end{proof}
According to the previous result, if $f:X \to \mathcal{P}(Z)$ is $K$-e-convex, we can reduce the images space of $f_K$ to
$$\mathcal{R}_K(Z):=\lbrace {A \in \mathcal{P}(Z) : A=\eco (A+K) } \rbrace.$$
\begin{remark}
Actually the set-valued functions $S_{(x^*,z^*)}$ and  $\bar {S}_{(x^*,z^*)}$ introduced in (\ref{defS}) map $X$ to  $\mathcal{R}_K(Z)$.
\end{remark}
In general, the Minkowski sum of two e-convex sets is not e-convex, see, for instance, \cite[Example 3.1]{GJR2003}. Therefore, we consider the modified Minkowski sum
$$A \boxplus B:=\eco(A+B),$$

\begin{proposition}
$(\mathcal{R}_K(Z), \boxplus)$ is a conmutative monoid with neutral element $K$.
\end{proposition}
\begin{proof}
Being evident the conmutativity of  the operation $\boxplus$ and the neutrality of $K$ in $\mathcal{R}_K(Z)$, let us prove the associative property. It will follow from the fact that, for $A,B \subseteq Z$ nonempty sets,
\begin{equation} \label{associative}
\eco (A+\eco B)=\eco (A+B).
\end{equation} 
To show this equality, is enough to see that $$\eco (A+\eco B) \subseteq \eco (A+B).$$
 According to \cite[Cor.2.1]{GR2006} \footnote{This result is established in a finite dimensional framework, and it can be easily extended to the locally convex spaces context.}, $$\eco A+\eco B \subseteq \eco (A+B),$$ hence
 $A+\eco B \subset\eco A+\eco B \subseteq \eco (A+B)$ and
$\eco (A+\eco B) \subseteq \eco (A+B)$.\\
Now let us take $A_1,A_2,A_3 \in \mathcal{R}_K(Z)$ nonempty sets (If at least one of them is the emptyset, the equality is trivial). We have, applying (\ref{associative}) when it is necessary,
$$A_1 \boxplus (A_2 \boxplus A_3)=\eco (A_1 +\eco (A_2 +A_3))=\eco (A_1+(A_2+A_3))=$$
$$\eco ((A_1+A_2)+A_3)=\eco (\eco (A_1+A_2)+A_3)=(A_1 \boxplus A_2) \boxplus A_3$$\qed
\end{proof}

Therefore, considering the modified Minkowski sum $\boxplus$ in $\mathcal{R}_K(Z)$ and the modified scalar product $$0A:=\eco K,$$ we conclude that $\mathcal{R}_K(Z)$ ia a partially ordered conlinear space.

 \begin{definition}
 For a set-valued function $f:X \to \mathcal{P}(Z)$, its \emph{K-e-convex hull} 
$$ \Keco f: X \to \mathcal{R}_K(Z),$$
is defined as the set-valued function whose graph is the e-convex hull of the $K$-epigraph of $f$, 
$$\epi_K (\Keco f)=\gph (\Keco f)=\eco (\epi_K f).$$

 \end{definition}
 
Let us observe that the $K$-e-convex hull of a set-valued function is (unique) well-defined. Moreover,  $\Keco f$ is the largest $K$-e-convex minorant of $f$, since denoting by $\mathcal{G}= \lbrace{ g : g \text{ is a $K$-e-convex minorant of } f} \rbrace$, then
$$ (\Keco f)(x)=\sup { \lbrace { g_K(x), g \in \mathcal G} \rbrace}, $$
for all $x \in X$.
\begin{remark}
In Remark \ref{rem1}, we saw how a set-valued function $f^s:X \to \mathcal {P}(\R)$, with $K=\R_+$, can be associated to any extended real-valued function $f: X \to \Ramp $, having the same epigraph, and $f$ is e-convex if and only if $f^s$ is $K$-e-convex. Nevertheless the $K$-e-convex hull of $f^s$ is not necessarily equal to the associated set-valued function of the e-convex hull of $f$, even their epigraphs cannot be equal, as the following example shows.
\end{remark}

\begin{example}
 Let $f:\R \to \R \cup \{+ \infty\}$ defined as
\begin{equation*}
f(x) :=\left\{
\begin{array}{ll}
+\infty & \text{ if } x<0,\\
1 & \text{if } x=0, \\
x^2 & \text{otherwise.}
\end{array}
\right.
\end{equation*}
Then
\begin{equation*}
f^s(x)=\left\{
\begin{array}{ll}
\emptyset & \text{ if } x<0,\\
\{1 \}& \text{if } x=0, \\
\{x^2\} & \text{otherwise}
\end{array}
\right.
\end{equation*}
and
\begin{equation*}
(\eco f)^s(x) =\left\{
\begin{array}{ll}
\emptyset & \text{ if } x<0,\\
\{0\} & \text{if } x=0, \\
\{x^2\} & \text{otherwise,}
\end{array}
\right.
\end{equation*}
but
\begin{equation*}
(\Keco f^s)(x) =\left\{
\begin{array}{ll}
\emptyset & \text{ if } x<0,\\
(0, +\infty ) & \text{if } x=0, \\
\left [x^2, +\infty \right ) & \text{otherwise.}
\end{array}
\right.
\end{equation*}
Moreover,
$$ \epi_K(\eco f)^s= \lbrace  {(x,z):x\geq 0, z\geq x^2 \rbrace}$$
and
$$ \epi_K  (\Keco f^s)= \gph (\Keco f^s)= \lbrace  {(x,z):x \geq 0, z\geq x^2 \rbrace}\setminus \{(0,0)\}.$$
We see that $ \epi_K (\Keco f^s) \subsetneq \epi ( \eco f)$. The reason is that the definition of the $K$-e-convex hull of a set-valued funtion as that function whose epigraph is the e-convex hull of the epigraph of the function does not work for an extended real valued function. As it happens in this example, the e-convex hull of an epigraph is not necessarily an extended real-valued function epigraph. For the function $f$, one has $ \eco (\epi f)=\lbrace  {(x,z):x \geq0, z\geq x^2 \rbrace} \setminus \{(0,0)\}$.
\end{example}

It is worthy to mention, for a better understanding of the next sections, that a set-valued function $f:X \to \mathcal{P}(Z)$ is $K$-convex ($K$-closed, $K$-e-convex, resp.) iff $\epi_K  (\Kconv f)$ ($\epi_K  (\Kcl f)$, $\epi_K ( \Keco f)$, resp.) is equal to $\epi_K f$, or, equivalently, $\Kconv f=f_K$, ($\Kcl f=f_K$, $\Keco f =f_K$, resp.). Hence, the behaviour of this kind of hulls differs from what is usually understood for a hull of any kind for an extended real-valued function, in the sense that, for instance, a $K$-e-convex set-valued function is not necessarily equal to its $K$-e-convex hull.\\
Finally, let us observe that, for all $x \in X$, $$(\Kcl  \conv f)(x) \leq_K (\Keco f)(x) \leq_K f(x),$$\\
and the first inequality can be strict. For instance,  in the previous example, we have

\begin{equation*}
(\Kcl \conv f^s)(x) =\left\{
\begin{array}{ll}
\emptyset & \text{ if } x<0,\\
\left[ x^2, +\infty \right ) & \text{otherwise.}
\end{array}
\right.
\end{equation*}

\section{Characterizing $K$-e-convex set-valued functions} \label{sec3}
It is a very well-known result  that a proper convex lower semicontinuos function $f: X \to \Ramp$, where $X$ is a locally convex space, is the pointwise supremum of its affine minorants, and \cite[Section 3]{H2009}  shows that this approach can be developed for set-valued functions with a suitable definition of a set-valued affine minorant.\\

Again in the scalar case,  in \cite{MLVP2011} different kinds of affine functions were defined, and proper e-convex functions were characterized as the pointwise supremum of  the different sets of its  affine minorants. Our aim in the following section is to characterize proper $K$-e-convex set-valued functions by means of its set-valued affine  minorants. We will define, like in \cite{MLVP2011}, the sets of $M_{f}$-affine, $\mathcal{C}_f$-affine and e-affine minorants of a set-valued function.\\

Previous definitions have been linked to the chosen convex cone $K$,  but in what follows, we assume that this cone is fixed, so it is not included in the next definitions, in order to avoid too long names.
\begin{definition}
Let $C \subseteq X$. A set-valued  function $a:X \to \mathcal{P}(Z)$ is \emph{C-affine} if there exist $x^* \in X^*$, $z^* \in K^*\setminus \{0\}$ and $\tilde z \in Z$ such that
\begin{equation} \label{c-affine}
a(x) =\left\{
\begin{array}{ll}
S_{(x^*,z^*)}(x)+\tilde z & \text{if } x\in C, \\
\emptyset & \text{otherwise.}
\end{array}
\right.
\end{equation}
 \end{definition}
\begin{remark}
It is easy to see that $a(x)+K=a(x)$, for all $x \in C$ (trivially if $x \notin C$), since $z \in a(x)$ if and only if $\langle x, x^* \rangle +\langle z- \hat z, z^*\rangle <0$, and $z^* \in K^*$. Then, a $C$-affine function maps any $x \in X$ to $a(x) \in \mathcal{P}^a_K(Z)$, i.e., $a:X \to \mathcal{P}^a_K(Z)$.
\end{remark}

In the next proposition we will use the fact that $C,D \subseteq X$ are e-convex iff $C\times D$ is e-convex. See \cite [Proposition 1.2] {RVP2011} for the proof in the finite dimensional case, which can easily be generalized to locally convex spaces.

\begin{proposition} \label{prop1}
Let  $a:X \to \mathcal{P}(Z)$ be a $C$-affine set-valued  function. Then if $C$ is e-convex, $a$ is $K$-e-convex. Moreover, in this case, $a:X \to \mathcal{R}_K(Z)$.
\end{proposition}
\begin{proof}
Let $C$ be e-convex in (\ref{c-affine}). We have
$$\epi_K a= \lbrace (x,z) : \langle x, x^* \rangle +\langle z, z^*\rangle < \langle \tilde z, z^* \rangle \rbrace \bigcap \lbrace C \times Z\rbrace,$$
which is an e-convex set, since it is the intersection of e-convex sets. It follows immediately that $a(x) \in  \mathcal{R}_K(Z)$, for all $x \in X$.
\qed
\end{proof}
\begin{remark}
It is well-known that if $C \subseteq X$ is a convex set and $T:X \to Y$ is a linear operator, (in this case, $X$ and $Y$ are real linear vector spaces), then $T(C)$ is also a convex set, and from this result one can derive that the effective domain of a $K$-convex set-valued function is convex,  also a well-known property. But this is not true, in general, either for e-convex sets or for $K$-e-convex functions. See, for instance, \cite[Example 2.4]{RVP2011}. 
\end{remark}
For a set-valued function $f:X \to \mathcal{P}(Z)$ we denote by $M_{f}=\eco( \dom f)$, and by $\mathcal{H}_f$ the set of all the $M_{f}$-affine minorants of $f$:
$$\mathcal{H}_f=\lbrace {a:X \to \mathcal{R}_K(Z): a \text{ is } M_{f} \text{-affine and }  a(x) \leq_K f(x), \text{ for all } x \in X  } \rbrace.$$
\begin{remark}\label{rem3}
Clearly, according to Proposition \ref{prop1}, every $M_{f}$-affine minorant of $f$ is $K$-e-convex. Moreover, $\mathcal{H}_f=\mathcal{H}_{f_K}$.
\end{remark}
\begin{remark}\label{remark1}
Let us see what happens with $\mathcal{H}_f$ is $f$ is an improper set-valued function.
If $f \equiv \emptyset$, then $\dom f = \emptyset$ and $\mathcal{H}_f=\{f\}$. On the other hand, if $f(\bar x)=Z$, for some $\bar x \in X$, in the case there exists $a \in \mathcal{H}_f$, it will follow that $S_{(x^*,z^*)}(\bar x)+\hat z=Z$  with $(x^*, z^*) \in X^* \times (K^* \setminus \{0\})$ and $\hat z \in Z$, and this is impossible, it must be  $\mathcal{H}_f=\emptyset$.
\end{remark}
The proof of the following lemma is similar to the scalar case (\cite [Lemma 7] {MLVP2011}). 
\begin{lemma}\label{lemma2}
Let  $f:X \to \mathcal{P}(Z)$ be a set-valued function. Then  $\mathcal{H}_f=\mathcal {H}_{K \text{-}{\eco f}}$.
\end{lemma}
\begin{lemma} \label{lemma7}
Let  $f:X \to \mathcal{P}(Z)$ be a set-valued function. If  K-$\eco f \equiv Z$ then necessarily $\mathcal{H}_f = \emptyset$.
\end{lemma}
\begin{proof}
If there exists $a \in \mathcal{H}_f$ as it is described in \eqref{c-affine}, then $a(x) \leq _K (\Keco f) (x)$, for all $x \in X$, and $a \equiv Z$, in particular, $a(0)=Z$, and it would imply that, taking into account that $z^* \neq 0$,  
\begin{equation*}
\lbrace z: \langle z,z^* \rangle < \langle \tilde z,z^* \rangle \rbrace=Z.\tag*{$\square$}
\end{equation*}
\end{proof}
\begin{theorem} \label{theorem1}
Let  $f:X \to \mathcal{P}(Z)$ be a set-valued function. The following statements are equivalent:
\begin{enumerate}[label=(\roman*)]
\item $\mathcal{H}_f \neq \emptyset$.
\item Either $K$-$\eco f$ is proper or $f \equiv \emptyset$.
\item $f$ has a proper $K$-e-convex minorant.
\end{enumerate}
\end{theorem}
\begin{proof}
$(i) \Rightarrow (ii)$ According to Remark \ref {remark1} and Lemma \ref{lemma2},  $(\Keco f)(x)\neq Z$ for all $ x \in X$, and therefore, in the case $K$-$\eco f$ is improper, it would be   $\dom (\Keco f) = \emptyset$, and then 
\begin{equation*}
\emptyset = (\Keco f) (x) \leq_K f(x),
\end{equation*}  
for all $x \in X$, which means that $f \equiv \emptyset$.\\
$(ii) \Rightarrow (iii)$ $\Keco f$ will be a proper $K$-e-convex minorant of $f$ in one case, and if $ f \equiv \emptyset$, any proper $K$-e-convex set-valued function will be a minorant of $f$.\\
$(iii) \Rightarrow (i)$ If $f \equiv \emptyset$ then $\mathcal{H}_f = \{f\}$. Now, consider the case where $\dom f \neq \emptyset$. Let $g$ be a proper $K$-e-convex minorant of $f$. Take a point $ x_0 \in \dom g$ and $z_0 \in Z$ such that $(x_0, z_0) \notin \epi_K g$. Since it is an e-convex set in $X \times Z$, by Proposition \ref{defeconv}, there exists $(x^* , z^*) \in (X^* \times Z^*)\setminus \{0\}$ such that
\begin{equation} \label{eq2}
\langle x_0, x^* \rangle + \langle z_0, z^* \rangle > \langle x, x^* \rangle + \langle z, z^* \rangle,
\end{equation}
for all $(x,z) \in \epi_K g$, hence $z^* \neq 0$.  In particular, for any fix point $z_1 \in g(x_0)$ and for all $k \in K$,
\begin{equation*}
\langle x_0, x^* \rangle + \langle z_0, z^* \rangle > \langle x_0, x^* \rangle + \langle z_1+k, z^* \rangle 
\end{equation*}
and $z^* \in K^* \setminus \{0\}$. We consider a point $\tilde z \in Z$ such that $\langle \tilde z, z^* \rangle \geq \langle x_0, x^* \rangle + \langle z_0, z^* \rangle$ and take
\begin{equation*}
a(x) =\left\{
\begin{array}{ll}S_{(x^*,z^*)}(x)+\tilde z & \text{if } x\in M_f, \\
\emptyset & \text{otherwise.}
\end{array}
\right.
\end{equation*}
Since $\epi_K f \subseteq \epi_K g$, we will have, for all $(x,z) \in \epi_K f$, $\langle x, x^* \rangle + \langle z, z^* \rangle < \langle \tilde z, z^* \rangle$, according to \eqref {eq2}, and $z-\tilde z \in S_{(x^*,z^*)}(x)$, then for all $x \in \dom f$, $f(x) - \tilde z \subseteq S_{(x^*,z^*)}(x)$ and $a \in \mathcal{H}_f $.\qed
\end{proof}
Taking into account that $\Keco f=f_K$ whenever $f$ is $K$-e-convex and $f \equiv \emptyset$ is equivalent to  $f_K \equiv \emptyset$, the following corollary comes directly.
\begin{corollary} \label{cor1}
Let  $f:X \to \mathcal{P}(Z)$ be a K-e-convex set-valued function. Then $\mathcal{H}_f \neq \emptyset$ if and only if  $ f_K$ is proper or $f_K \equiv \emptyset$.
\end{corollary}

We recall that, if  $f:X \to \mathcal{P}(Z)$ is a set-valued function and $\lbrace a_i: i \in I \rbrace$ is an arbitrary collection of set-valued functions where $a_i : X \to \mathcal{P}(Z)$, for all $i \in I$, saying that $f$ is the pointwise supremum of  $\lbrace a_i: i \in I \rbrace$,  $$f(x)+K=\sup_{ I} \lbrace { a_i (x)+K, \leq_K} \rbrace,$$  means that $f(x)+K=\bigcap_{ I} \lbrace { a_i (x)+K} \rbrace$, for all $x  \in X$. It is clear that we can also take a collection  $\lbrace a_i: i \in I \rbrace$ where  $a_i : X \to \mathcal{P}^a_K(Z)$.
\begin{remark}
Let us observe that  $f$ is the pointwise supremum of  $\lbrace a_i: i \in I \rbrace$ if $\epi_K f=\bigcap_{ I}{ \epi_K a_i }$.
\end{remark}
\begin{theorem} \label{theorem2}
Let  $f:X \to \mathcal{P}(Z)$ be a set-valued function. The following statements are equivalent:
\begin{enumerate}[label=(\roman*)]
\item $f$ is the pointwise supremum of its $M_{f}$-affine minorants.
\item $ f$ is K-e-convex  and either $f_K$ proper, $f_K \equiv \emptyset$ or $f_K \equiv Z$.
\end{enumerate}
\end{theorem}
\begin{proof}
$(i) \Rightarrow (ii)$ In the case  $\mathcal{H}_f=\emptyset$,  then  $f_K \equiv Z$ (recall that the supremum of an empty collection of sets in $\mathcal{P}^{a}_K(Z)$ is $Z$). Clearly, $\epi_K f=X \times Z$, which is e-convex.\\
 Otherwise, if $\mathcal{H}_f\neq \emptyset$, since $ \epi_K f = \bigcap_{ \mathcal{H}_f}  \epi_K a $ and it is an e-convex set (recall Remark \ref{rem3}), then $f$ is $K$-e-convex, and by Corollary \ref{cor1}, either $ f_K$ is proper or $f_K \equiv \emptyset$.\\
 $(ii) \Rightarrow (i)$  If $f_K \equiv Z$, by Remarks \ref{rem3} and \ref{remark1}, $\mathcal{H}_f = \emptyset$ and $f$ is the pointwise supremum of an empty family of minorants.\\
  If $f_K \equiv \emptyset$, then $\mathcal{H}_f= \{ f\}$. Finally, assume that $f$ is $K$-e-convex, with $f_K$ proper. According again to Corollary \ref{cor1}, $\mathcal{H}_f \neq  \emptyset$. 
 Since $a(x) \leq_K f(x)$, for all $x \in X, a \in \mathcal{H}_f$,  it follows $ \epi_K f \subseteq \bigcap_{ \mathcal{H}_f}  \epi_K a $. For the converse inclusion, we will show that if $(x_0, z_0) \notin \epi_K f$, then there exists $a \in \mathcal{H}_f$ such that  $(x_0, z_0) \notin \epi_K a$. \\
We can assume that $x_0 \in \mathcal M_f$, otherwise $a(x_0)= \emptyset$, for all $a \in \mathcal{H}_f$.\\
 Since $\epi_K f$ is e-convex, according to Proposition \ref{defeconv}, we can find $(x^* , z^*) \in (X^* \times Z^*)\setminus \{0\}$ such that
\begin{equation} \label{eq3}
\langle x_0, x^* \rangle + \langle z_0, z^* \rangle > \langle x, x^* \rangle + \langle z, z^* \rangle,
\end{equation}
for all $(x,z) \in \epi_K f$. We ensure that $z^* \in K^* \setminus \{0\}$.\\
Firstly, we will see that $z^* \neq 0$. If not, $\langle x_0, x^* \rangle  > \langle x, x^* \rangle,$
for all $x \in \dom f$, and it follows that $x_0 \notin M_f$.\\
Now, if $z^* \notin K^*$, it would exist $k_0 \in K$ such that $\langle k_0, z^* \rangle >0$, and considering that, for all $(x,z) \in \epi_K f$ and $\lambda >0$, $(x, z+ \lambda k_0) \in \epi_K f$, inequality (\ref{eq3}) will not hold for $\lambda$ large enough, which leads us to conclude that $z^* \in K^* \setminus \{0\}$.
 We consider now  a point $\tilde z \in Z$ such that $\langle \tilde z, z^* \rangle = \langle x_0, x^* \rangle + \langle z_0, z^* \rangle$ and take
\begin{equation*}
a(x) =\left\{
\begin{array}{ll}
S_{(x^*,z^*)}(x)+\tilde z & \text{if } x\in M_f, \\
\emptyset & \text{otherwise.}
\end{array}
\right.
\end{equation*}
 It is easy to check that $a \in \mathcal{H}_f$, nevertheless, 
 \begin{equation*}
 \langle x_0, x^* \rangle + \langle z_0-\tilde z, z^* \rangle  =0,
 \end{equation*}
 and $z_0  \notin a(x_0)$, hence $(x_0, z_0) \notin  \epi_K a$. 
  \qed
\end{proof}
\begin{corollary} \label{cor2}
Let  $f:X \to \mathcal{P}(Z)$ having a proper K-e-convex minorant. Then, for all $x \in X$,
\begin{equation*}
(K \text{-}\eco f)(x)=\sup_{ \mathcal{H}_f}\{ a(x), \leq_K  \}. 
\end{equation*}
\end{corollary}
\begin{proof}
By Theorem \ref{theorem1}, either $ \Keco f$ is proper or $ f \equiv \emptyset$. In the first case, taking into account that  $ \Keco f=( \Keco f)_K$,  by Theorem \ref{theorem2},  $ \Keco f$ is the pointwise supremum of its $M_{\Keco f}$-affine minorants, which is the same set as the set of the $M_{f}$-affine minorants of $f$, according to Lemma \ref{lemma2}.\\
In the case $f \equiv \emptyset$, $f$ is $K$-e-convex, hence $\Keco f =f_K= f$ and moreover $\mathcal{H}_f=\{f \}$. \qed
\end{proof}

\begin{definition}
Let $\mathcal{C}$ be the set of all e-convex subsets in $X$. We say that a set-valued function $a: X \to \mathcal{R}_K(Z)$ is $\mathcal{C}$\emph{-affine} if there exists $C \in \mathcal{C}$ such that $a$ is $C$-affine.

For a set-valued function $f:X \to \mathcal{P}(Z)$, we will denote by $\mathcal{C}_f$ the set of all its  $\mathcal{C}$-affine minorants.
\end{definition}
\begin{theorem} \label{theorem3}
Let  $f:X \to \mathcal{P}(Z)$. The following statements are equivalent:
\begin{enumerate}[label=(\roman*)]
\item $f$ is the pointwise supremum of its $\mathcal{C}$-affine minorants.
\item $ f$ is $K$-e-convex and either $f_K$ is proper, $f_K \equiv \emptyset$ or $f_K \equiv Z$.
\end{enumerate} 
\end{theorem}
\begin{proof}
$(i) \Rightarrow (ii)$
It is clear that $f$ is $K$-e-convex. Now, if $\mathcal{C}_f=\emptyset$, then $f_K \equiv Z$. In the case  $\mathcal{C}_f \neq \emptyset$, if $f_K \not \equiv \emptyset$, it must be proper, otherwise there would exist $x \in \dom f_K$ satisfying $f_K(x)=Z$ and hence for all $a \in \mathcal{C}_f$ it would be $a(x)=Z$ and it is not possible (we can use the same reasoning than the one used in Remark \ref{remark1} for $\mathcal{H}_f$). \\
$(ii) \Rightarrow (i)$ In the case $f_K \equiv Z$, we will have $\mathcal{C}_f = \emptyset$  and $f$ is the pointwise supremum of an empty family.
On the other hand, in the cases $f_K$ proper or $f_K \equiv \emptyset$, we will have $ \emptyset \neq \mathcal{H}_f  \subseteq \mathcal{C}_f$, and by Theorem \ref{theorem2}, we obtain
\begin{equation*}
\epi_K f =  \bigcap_{ \mathcal{H}_f}  \epi_K a \supseteq \bigcap_{\mathcal{C}_f}  \epi_K a \supseteq \epi_K f. \tag*{$\square$}
\end{equation*} 
\end{proof}
\begin{definition}
We say that a function $a: X \to \mathcal{R}_K(Z)$ is \emph {e-affine} if 
\begin{equation*}
a(x) =\left\{
\begin{array}{ll}
S_{(x^*,z^*)}(x)+\tilde z & \text{if } \langle x, y^* \rangle < \alpha, \\
\emptyset & \text{otherwise,}
\end{array}
\right.
\end{equation*}
for some $x^*,y^* \in X^*, z^*\in K^*\setminus \{0\}$, $\tilde z \in Z$ and $\alpha \in \R$.
For a set-valued function $f:X \to \mathcal{P}(Z)$, we will denote by $\mathcal{E}_f$ the set of all the e-affine minorants of $f$.
\end{definition}
\begin{theorem} \label{theorem4}
Let  $f:X \to \mathcal{P}(Z)$ a set-valued function. The following statements are equivalent:
\begin{enumerate}[label=(\roman*)]
\item $f$ is the pointwise supremum of its e-affine minorants.
\item $f$ is $K$-e-convex and either $ f_K$ is proper,   $f_K \equiv \emptyset$ or $f_K\equiv Z$.
\end{enumerate}
\end{theorem}
\begin{proof}
$(i) \Rightarrow (ii)$ The proof is similar to the proof of $(i) \Rightarrow (ii)$ in Theorem \ref{theorem3}. Just change $\mathcal{C}_f$ by $\mathcal{E}_f$.\\
$(ii) \Rightarrow (i)$ In the case  $f_K \equiv Z$ we have $\mathcal{E}_f=\emptyset$.  Let us assume then that  either $ f_K$ is proper or $f_K\equiv Z$. By Corollary \ref{cor1} we have $ \emptyset \neq \mathcal{H}_f \subseteq \mathcal{C}_f$ and by Theorem \ref{theorem3}, $\bigcap_{\mathcal{C}_f}  \epi_K a = \epi_K f$. Now, take any function $a \in \mathcal{C}_f$, and write $\dom a = \bigcap_{ T_a} H^-_t$ expressing this e-convex set as an intersection of open half-spaces. We define, for all $t \in T_a$, 
\begin{equation*}
a_t(x) =\left\{
\begin{array}{ll}
S_{(x^*,z^*)}(x)+\tilde z  & \text{if } x \in H^-_t, \\
\emptyset & \text{otherwise,}
\end{array}
\right.
\end{equation*}
where $x^*, z^*$ and $\tilde z$ are the data defining the function $a$. It is easy to check that $\epi_K a= \bigcap_{ T_a} \epi_K a_t$ and then, denoting by $S =\bigcup_{\mathcal{C}_f} \lbrace a_t, t \in T_a\rbrace$, we have $S \subset \mathcal{E}_f$ and so $\mathcal{E}_f \neq \emptyset$ and
\begin{equation*}
 \epi_K f = \bigcap_{ \mathcal{C}_f}\left \{ \bigcap_{T_a}  \epi_K a_t \right \} =\bigcap_{ S} \epi_K a \supseteq \bigcap_{\mathcal{E}_f} \epi_K a \supseteq \epi_K f.\tag*{$\square$} 
 \end{equation*}
\end{proof}
\begin{remark}
Let us observe that, as a result of the above proof, $\mathcal{E}_f \neq \emptyset$ whenever $f$ is $K$-e-convex and $f_K$ is proper.
\end{remark}
\begin{corollary}
Let  $f:X \to \mathcal{P}(Z)$ be a set-valued function having a proper $K$-e-convex minorant. Then $K$-$ \eco f$ is the pointwise supremum of its e-affine minorants.
\end{corollary}
\begin{proof}
The proof of $\mathcal{E}_f = \mathcal{E}_{\Keco f}$ is similar to that one of $\mathcal{H}_f = \mathcal{H}_{\Keco f}$ (we recall Lemma \ref{lemma2}). Now, according to Theorem \ref{theorem1}, either $\Keco f=(\Keco f)_K$ is proper or $f \equiv \emptyset$. In the first case, by Theorem \ref{theorem4} it follows that $\Keco f$ is the pointwise supremum of its e-affine minorants. On the other hand, if $f \equiv \emptyset$, it would be $f=f_K=\Keco f$ with $\mathcal{E}_f =\lbrace f \rbrace$. \qed
\end{proof}
\section{C-conjugating set-valued functions}\label{sec4}
\label{section:C-conjugation}
We will generalize the conjugation pattern for extended real-valued functions, described in \cite{MLVP2011}, suitable for e-convex functions. It is based on the generalized convex conjugation theory introduced by Moreau \cite{Mor1970}. We give a brief description of it.  Let us consider the space $W:=X^{\ast }\times X^{\ast }\times \R$ with \emph{the coupling functions }$c:X\times W\rightarrow \overline{\R}$ and $c^{\prime }:W\times X\rightarrow \overline{\R}$ given by
\begin{equation*}
c(x,(x^{\ast },y^{\ast },\alpha ))=c^{\prime }\left( (x^{\ast },y^{\ast},\alpha ),x\right) :=\left\{
\begin{array}{ll}
\left\langle x,x^{\ast }\right\rangle & \text{if }\left\langle x,y^{\ast}\right\rangle <\alpha,\\
+\infty & \text{otherwise.}
\end{array}
\right.
\end{equation*}
Given two functions $f:X\rightarrow \overline{\R}$ and $g:W\rightarrow \overline{\R}$, the \emph{$c$-conjugate} of $f$, $f^{c}:W\rightarrow \overline{\R}$, and the \emph{c$^{\prime }$-conjugate} of $g$, $g^{c^{\prime}}:X\rightarrow \overline{\R}$, are defined
\begin{eqnarray}\label{eq6}
f^{c}(x^{\ast },y^{\ast },\alpha)& :=& \supp_{x\in X}\left\{ c(x,(x^{\ast},y^{\ast },\alpha ))-f(x)\right\},\\
 g^{c^{\prime }}(x) & :=& \supp_{(x^{\ast },y^{\ast },\alpha )\in W}\left\{c^{\prime }\left( (x^{\ast },y^{\ast },\alpha ),x\right) -g(x^{\ast},y^{\ast },\alpha )\right\},
\end{eqnarray}
respectively, with the conventions $\left( +\infty \right) +\left( -\infty \right)=\left( -\infty \right) +\left( +\infty \right) =\left( +\infty \right)-\left( +\infty \right)$ $=\left( -\infty \right) -\left( -\infty \right)=-\infty.$\\
The  counterpart of the Fenchel-Moreau theorem for e-convex functions derives immediately from \cite[Corollary 6.1]{ML2005}, and a function  $f:X\en \R\cup \left\{+\infty\right\}$  is e-convex iff $f^{c c^{\prime}}=f$.

We can rewrite (\ref{eq6}) as
\begin{equation*}
f^{c}(x^{\ast },y^{\ast },\alpha) =\left\{
\begin{array}{ll}
 -\inf_{x\in X}\left\{ f(x)+ \langle -x, x^* \rangle \right\} & \text{if} \dom f \subseteq {H}^-_{y^*,\alpha}\\
 + \infty & \text{otherwise,}
 \end{array}
 \right.
 \end{equation*}
  for all $(x^{\ast },y^{\ast },\alpha) \in W$. This formula inspires the generalization of c-conjugacy to set-valued functions. As in the previous section, we assume that the convex cone $K$ is fix. 
\begin{definition} \label{def10}
Let $f:X \to \mathcal{P}(Z)$ be a set-valued function. We define the \emph{c-conjugate} of $f$ (respect to $K$), $f^c:X^* \times X^* \times ( K^* \setminus \{0\}) \times \R \to  \mathcal{P}(Z)$ as
\begin{equation*}
f^c (x^*, y^*, z^*, \alpha):=\left\{
\begin{array}{ll}
-\eco \left [ \bigcup_{ X} \lbrace f(x)+ \bar S_{(x^*,z^*)}(-x) \rbrace \right ] & \text{if } \dom f \subseteq {H}^-_{y^*,\alpha}\\
 \emptyset & \text{otherwise.}
 \end{array}
 \right.
 \end{equation*}

For a set-valued function $g:X^* \times X^* \times (K^* \setminus \{0\}) \times \R  \to \mathcal{P}(Z)$, its  \emph{c$^{\prime }$-conjugate} (respect to $K$), $g^{c'}:X \to \mathcal{P}(Z)$ is defined, if $ x \in{H}^-_{y^*,\alpha}$, for all $(x^*, y^*, z^*, \alpha) \in\dom g$:
\begin{equation*}
g^{c'} (x):=\bigcap_{\dom g} { \left \lbrace \bar S_{(x^*,z^*)}(x) - g (x^*, y^*, z^*, \alpha) \right \rbrace },
\end{equation*}
and $g^{c'} (x):=\emptyset$ otherwise.
 \end{definition}
 
 \begin{remark}
 Actually, $$ \inf \lbrace \lbrace f(x)+ \bar S_{(x^*,z^*)}(-x) \rbrace_X, \leq_K \rbrace=\bigcup_{ X} \lbrace f(x)+ \bar S_{(x^*,z^*)}(-x) \rbrace \text{   and}$$ 
 $$ \sup \left \lbrace\lbrace \bar S_{(x^*,z^*)}(x) - g (x^*, y^*, z^*, \alpha) \rbrace_{\dom g}, \leq_K \right  \rbrace = 
 \bigcap_{\dom g} { \left \lbrace \bar S_{(x^*,z^*)}(x) - g (x^*, y^*, z^*, \alpha) \right \rbrace }$$
 in the previous definitions, since $\bar S_{(x^*,z^*)}(-x)+K=\bar S_{(x^*,z^*)}(-x)$, for all $x \in X$.
 \end{remark}
\begin{definition}\label{definition1}
The function $\sigma_f : X^* \times X^* \times Z^* \times \R \to  \Ramp $ defined as
 \begin{equation*}
\sigma_f (x^*, y^*, z^*, \alpha):=\left\{
\begin{array}{ll}
\sup_{ \epi_K f}{ \lbrace \langle x, x^* \rangle + \langle z, z^* \rangle \rbrace}& \text{ if } \dom f \subseteq {H}^-_{y^*,\alpha},\\
 +\infty & \text{otherwise.}
 \end{array}
 \right.
 \end{equation*}
is called the \emph{support function} of the $K$-epigraph of the set-valued function $f:X \to \mathcal{P}(Z)$ relative to open half-spaces.
\end{definition}
 \begin{remark}
  Let $f:X \to \mathcal{P}(Z)$ be a set-valued function and let us consider the indicator function of $\epi_K f$, $\delta_{\epi_K f} :X \times Z \to \Ramp$,
 \begin{equation*}
\delta_{\epi_K f} (x,z)=\left\{
\begin{array}{ll}
0 & \text{if } z \in f(x)+K, \\
+\infty & \text{otherwise.}
\end{array}
\right.
\end{equation*}
 It is easy to check that $ \delta_{\epi_K f}^c (x^*, z^*,  y^*,0, \alpha) =\sigma_f (x^*, y^*, z^*, \alpha)$.
  \end{remark}
  \begin{remark}\label{remarkdomsigma}
  Let us observe that, if $f$ is $K$-e-convex and $f_K \not \equiv Z$, $\epi_K f$ is a proper e-convex subset of $X \times Z$ and hence it can be expressed as the intersection of a nonempty family of open half-spaces, being nontrivial at least one of them. Denote one of these half-spaces by $H= \lbrace (x,z) \in X \times Z:\langle x,x^*\rangle + \langle z, z^* \rangle <\beta \rbrace$. Since $\epi_K f \subseteq H$, we will have that  $\sigma_f (x^*, 0, z^*, \alpha) \leq \beta <+\infty$, for all $\alpha >0$ and $\dom \sigma_f \neq \emptyset$.
  \end{remark}
  
  \begin{definition}
  Let $C\subset X \times Z$ be a nonempty set. We define the function $\eta_C:X^* \times Z^* \to \lbrace 0,1\rbrace$ as
   \begin{equation*}
\eta_C(x^*,z^*)=\left\{
\begin{array}{ll}
0 & \text{if } \langle  x, x^* \rangle + \langle  z, z^* \rangle <\sup_C{ \lbrace \langle x, x^* \rangle + \langle z, z^* \rangle \rbrace} \text{ for all } (x, z) \in C, \\
1 & \text{otherwise.}
\end{array}
\right.
\end{equation*}
  \end{definition}
 \begin{lemma}\label{lemma3}
  Let  $f:X \to \mathcal{P}(Z)$ be a set-valued function. Then, for all  $(x^*, y^*, z^*, \alpha) \in X^* \times X^* \times (K^* \setminus \{0\}) \times \R $,
  \begin{equation*}
  -f^c (x^*, y^*, z^*, \alpha)= \left\{
  \begin{array}{ll}
  \lbrace z \in Z : \langle z, z^* \rangle<\sigma_f (x^*, y^*, z^*, \alpha) \rbrace & \text{ if } \eta_{\epi_K f}(x^*,z^*)=0,\\
  \lbrace z \in Z : \langle z, z^* \rangle\leq \sigma_f (x^*, y^*, z^*, \alpha) \rbrace & \text{ if } \eta_{\epi_K f}(x^*,z^*)=1.
  \end{array}
  \right.
  \end{equation*}
  Consequently, $\dom f^{c}=\dom \sigma_f$.
 \end{lemma}
 \begin{proof}
 Take any point $(x^*, y^*, z^*, \alpha) \in X^* \times X^* \times (K^* \setminus \{0\}) \times \R$.\\
 In first place, the equality holds trivially in the case $\dom f \not \subset H^-_{y^*, \alpha}$, because $\sigma_f (x^*, y^*, z^*, \alpha)=+\infty$ and $ -f^c (x^*, y^*, z^*, \alpha)= Z$ (recall that we set $-\emptyset=Z$).\\
 Hence, let us suppose that $\dom f \subset H^-_{y^*, \alpha}$.
 We denote  $$D_0=  \lbrace z \in Z : \langle z, z^* \rangle<\sigma_f (x^*, y^*, z^*, \alpha) \rbrace,$$  $$D_1=  \lbrace z \in Z : \langle z, z^* \rangle \leq \sigma_f (x^*, y^*, z^*, \alpha) \rbrace.$$\\
 In the case $ \eta_{\epi_K f}(x^*,z^*)=0$, take any point $x \in \dom f$, $z_1 \in f(x)$ and $z_2 \in \bar S_{(x^*,z^*)}(-x)$.  Then $$\langle z_1+z_2, z^* \rangle<\sigma_f (x^*, y^*, z^*, \alpha),$$ and $z_1+z_2 \in D_0$, which means that, for all $x \in \dom f$, $$f(x)+ \bar S_{(x^*,z^*)}(-x) \subset D_0,$$
and $D_0$ an e-convex set, implying that  $-f^c (x^*, y^*, z^*, \alpha) \subset D_0$.\\
For the converse inclusion, take $z \in D_0$ and $(\hat x, \hat z) \in \epi_K f$ verifying that $$\langle z, z^* \rangle <\langle \hat x, x^* \rangle + \langle \hat z, z^* \rangle. $$
 Then, since $\hat z=z_1+k$, for some $z_1 \in f(\hat x)$ and $k \in K$, we have
 $$\langle -\hat x, x^* \rangle + \langle z-z_1, z^* \rangle<\langle k, z^* \rangle \leq 0,$$
 and $z-z_1 \in  \bar S_{(x^*,z^*)}(-\hat x)$, obtaining $D_0 \subset -f^c (x^*, y^*, z^*, \alpha)$.
 The proof is similar for the case $ \eta_{\epi_K f}(x^*,z^*)=1$ and $D_1$; the inclusion $-f^c (x^*, y^*, z^*, \alpha) \subset D_1$ can be shown analogously, whereas the converse inclusion holds if we take a point  $(\hat x, \hat z) \in \epi_K f$ verifying that $$\langle z, z^* \rangle \leq \langle \hat x, x^* \rangle + \langle \hat z, z^* \rangle. $$\\
 Finally, we have that  $(x^*, y^*, z^*, \alpha) \in \dom f^c$ iff  $-f^c (x^*, y^*, z^*, \alpha)\neq Z$, which means that $\sigma_f (x^*, y^*, z^*, \alpha) <+\infty$, i.e.,  $(x^*, y^*, z^*, \alpha) \in \dom \sigma_f$.
  \qed
 \end{proof}

\begin{lemma}\label{lemma4}
 Let $f:X \to \mathcal{P}(Z)$ a set-valued function. Then, if $x \in  \dom f^{cc'}$,
 \begin{equation} \label{eq8}
  \begin{array}{ll}
 $$f^{cc'} (x)=\bigcap_{\dom \sigma_f} \lbrace  z \in Z :& \langle x, x^* \rangle+\langle z, z^* \rangle<\sigma_f (x^*, y^*, z^*, \alpha), \eta_{\epi_K f}(x^*,z^*)=0,\\
   &\langle x, x^* \rangle+\langle z, z^* \rangle \leq \sigma_f (x^*, y^*, z^*, \alpha), \eta_{\epi_K f}(x^*,z^*)=1 \rbrace .
 \end{array}
 \end{equation}
  Consequently,  $f^{cc'} (x)$ is an e-convex set in $Z$, for all $x \in X$, and $f^{cc'}: X \to  \mathcal{R}_K(Z)$ is a $K$-e-convex set-valued function. Moreover, $ f(x)\subseteq f^{cc'}(x) $ and then $f_K (x) \subseteq f^{cc'} (x)$,  for all $x \in X$. 
\end{lemma}
\begin{proof}
Take $\bar x \in \dom f^{cc'}$  and denote by

 \begin{equation*}
  \begin{array}{ll}
 $$D=\bigcap_{\dom \sigma_f} \lbrace  z \in Z :& \langle \bar x, x^* \rangle+\langle z, z^* \rangle<\sigma_f (x^*, y^*, z^*, \alpha), \eta_{\epi_K f}(x^*,z^*)=0,\\
   &\langle \bar x, x^* \rangle+\langle z, z^* \rangle \leq \sigma_f (x^*, y^*, z^*, \alpha), \eta_{\epi_K f}(x^*,z^*)=1 \rbrace .
 \end{array}
 \end{equation*}
Take now $(\bar x^*,\bar  y^*, \bar z^*,\bar  \alpha) \in \dom \sigma_f$. Name
 $$ D_0=\lbrace z \in Z:  \langle \bar x, \bar x^* \rangle+\langle z, \bar z^* \rangle<\sigma_f (\bar x^*, \bar y^*, \bar z^*, \bar \alpha) \rbrace,$$
 $$D_1=\lbrace z \in Z:  \langle \bar x, \bar x^* \rangle+\langle z, \bar z^* \rangle \leq \sigma_f (\bar x^*, \bar y^*, \bar z^*, \bar \alpha)  \rbrace.$$
   
From the definition of the {c$^{\prime }$-conjugate} function, and taking into account that, according to Lemma \ref{lemma3},  $\dom f^c=\dom \sigma_f$, it holds
\begin{equation*}
f^{cc'} (\bar x)=
\bigcap_{\dom \sigma_f} { \left [ \bar S_{(x^*,z^*)}( \bar x) - f^c (x^*, y^*, z^*, \alpha) \right ]} .
 \end{equation*}
Take $\bar z \in f^{cc'} (\bar x)$ and suppose  $\eta_{\epi_K f}(\bar x^*,\bar z^*)=0$. Since $(\bar x^*, \bar y^*,\bar  z^* ,\bar  \alpha) \in \dom \sigma_f$, we have $\bar z \in \bar S_{(\bar x^*,\bar z^*)}( \bar x) - f^c (\bar x^*, \bar y^*, \bar z^*, \bar \alpha)$ and  $\bar z= z_1 + z_2$, where $z_1 \in S_{(\bar x^*,\bar z^*)}( \bar x)$ and $z_2 \in - f^c (\bar x^*, \bar y^*, \bar z^*, \bar \alpha)$. Therefore
$$ \langle \bar x, \bar x^* \rangle+\langle z_1, \bar z^* \rangle \leq 0,$$
and, by Lemma \ref {lemma3}, $$\langle z_2, \bar z^* \rangle<\sigma_f (\bar x^*, \bar y^*, \bar z^*, \bar \alpha).$$
Then $ \langle \bar x, \bar x^* \rangle+\langle \bar z, \bar z^* \rangle <\sigma_f (\bar x^*, \bar y^*, \bar z^*, \bar \alpha)$ and $\bar z \in D_0$. Analogously, it can be shown that, if  $\eta_{\epi_K f}(\bar x^*,\bar z^*)=1$, $f^{cc'} (\bar x) \subseteq D_1 $. Then $f^{cc'} (\bar x) \subseteq D$.\\
For the converse inclusion, take $\bar z \in D$ and $(\bar x^*, \bar y^*, \bar z^*, \bar \alpha) \in  \dom \sigma_f=\dom f^c$. In the case $\eta_{\epi_K f}(\bar x^*,\bar z^*)=0$, we denote by 
$$\beta = \sigma_f (\bar x^*, \bar y^*, \bar z^*, \bar \alpha)-\langle \bar x, \bar x^* \rangle-\langle \bar z, \bar z^* \rangle>0.$$
Choose $z_2 \in Z$ such that $-\beta < \langle \bar x, \bar x^* \rangle+\langle z_2, \bar z^* \rangle \leq 0$, 
and $\bar z=(\bar z-z_2)+z_2$, with $z_2 \in \bar S_{(\bar x^*,\bar z^*)}( \bar x)$ and, according to Lemma \ref{lemma3},  $\bar z-z_2 \in  -f^c(\bar x^*, \bar y^*, \bar z^*, \bar \alpha)$.\\
In the case $\eta_{\epi_K f}(\bar x^*,\bar z^*)=1$, we can follow the same steps than in the previous proof, choosing $z_2 \in Z$ such that $\langle \bar x, \bar x^* \rangle+\langle z_2, \bar z^* \rangle =0$.\\
Hence, if $\bar x \in\dom f^{cc'}$, then $f^{cc'} (\bar x)=D$.\\
It is clear that, for all $x \in X$, $f^{cc'} (x)$ is an e-convex set in $Z$. Moreover, $\epi_K f^{cc'}$ is the intersection of the epigraphs of all the $K$-e-convex set-valued functions
$$a_0(x)=  \lbrace  z \in Z : \langle x, x^* \rangle+\langle z, z^* \rangle<\sigma_f (x^*, y^*, z^*, \alpha)\rbrace \text { if }  \eta_{\epi_K f}(x^*,z^*)=0,$$
$$a_1(x)=  \lbrace  z \in Z : \langle x, x^* \rangle+\langle z, z^* \rangle\leq \sigma_f (x^*, y^*, z^*, \alpha) \rbrace  \text { if }  \eta_{\epi_K f}(x^*,z^*)=1,$$
for all $(x^*, y^*, z^*, \alpha) \in \dom \sigma_f$, and $f^{cc'}$ is $K$-e-convex. \\
Is is evident also from (\ref{eq8}) that $f^{cc'}(x)+K \subseteq f^{cc'}(x)$, for all $x \in X$, therefore $\eco (f^{cc'}(x)+K) \subseteq f^{cc'}(x)$, for all $x \in X$, and $f^{cc'}: X \to  \mathcal{R}_K(Z)$. The inclusions $ f(x) \subseteq f^{cc'}(x)$ and $f_K (x) \subseteq f^{cc'} (x)$, for all $x \in X$,  are trivial.
 \qed
\end{proof}

Next we will give the biconjugation theorem.
\begin{theorem}\label{theo5}
Let $f:X \to \mathcal{P}(Z)$ be a set-valued function. Then it is $K$-e-convex if and only if $f^{cc'} (x)=f_K(x)$, for all $x \in X$.
\end{theorem}
\begin{proof}  
Let us assume that $f$ is $K$-e-convex. If $f_K \equiv Z$, from Lemma \ref{lemma4}, we have $f_K=f^{cc'}$. In the case $f \equiv \emptyset$, it will be $\sigma_f \equiv -\infty$ and again applying Lemma \ref{lemma4}, $f^{cc'} \equiv \emptyset$. Hence, if either $f_K \equiv Z$ or $f \equiv \emptyset$, it follows that $f^{cc'}=f_K$. \\
Suppose that $f_K \not \equiv Z$ and $\dom f \neq \emptyset$. According to  Lemma \ref {lemma4}, it is enough to prove that $ f^{cc'}(x) \subseteq f_K(x) $, for all $x \in X$. In first place, let us observe that, 
according again to Lemma \ref{lemma4}, $ \emptyset \neq \dom f  \subseteq \dom f^{cc'}$ and, moreover, recalling Remark \ref{remarkdomsigma}, $\dom \sigma_f \neq \emptyset$, and for all $x \in \dom f^{cc'}$, $f^{cc'}(x)\neq Z$ and $f^{cc'}$ is proper. Since  $ f_K(x)  \subseteq f^{cc'}(x)$, for all $x \in X$, and  $f_K \not \equiv \emptyset$, $f_K$ is also proper.\\
 Applying Theorem \ref{theorem4} to both functions, $f$ and $f^{cc'}$, we have $\mathcal {E}_f \neq \emptyset$, $\mathcal {E}_{f^{cc'}} \neq \emptyset$ (see the proof of Theorem \ref{theorem4}) and
$$\epi_K f = \bigcap_{\mathcal{E}_f}\epi_K a \text {     and    } \epi_K f ^{cc'}= \bigcap_{\mathcal{E}_{f^{cc'}}}\epi_K a.$$
We will show that $\mathcal{E}_f \subseteq \mathcal{E}_{f^{cc'}}$. Take any e-affine function $a \in \mathcal{E}_f$,
\begin{equation*}
a(x) =\left\{
\begin{array}{ll}
S_{(x^*,z^*)}(x)+\tilde z & \text{if }  \langle x, y^* \rangle < \alpha, \\
\emptyset & \text{otherwise.}
\end{array}
\right.
\end{equation*}
Then $f(x) \subseteq a(x)$, for all $x \in X$, and it holds, for all $x \in \dom f \subset H^-_{y^*, \alpha}$,
$$f(x) \subseteq S_{(x^*,z^*)}(x)+\tilde z.$$
Taking into account that $\bar S_{(x^*,z^*)}(- x)+S_{(x^*,z^*)}(x)  \subseteq S_{(x^*,z^*)}(0)$, we will obtain, for all $x \in \dom f$,
$$f(x) + \bar S_{(x^*,z^*)}(- x)  \subseteq S_{(x^*,z^*)}(0)+\tilde z.$$
Clearly $ S_{(x^*,z^*)}(0)+\tilde z$ is an e-convex set, hence
\begin{equation} \label{eq4}
-f^c (x^*, y^*, z^*, \alpha) \subseteq S_{(x^*,z^*)}(0)+\tilde z,
\end{equation}
and $(x^*, y^*, z^*, \alpha) \in \dom f^c$.\\
Now,  if for some point $\bar x \in X$, $f^{cc'}(\bar x)= \emptyset$, then  $f^{cc'}(\bar x) \subseteq a(\bar x)$ and $a(\bar x) \leq_K  f^{cc'}(\bar x)$. Then, assuming that $f^{cc'}(\bar x) \neq \emptyset$, from (\ref{eq4}) and Definition \ref{def10} we conclude that 
$$f^{cc'}( \bar x) \subseteq  \bar S_{(x^*,z^*)}(\bar x) +  S_{(x^*,z^*)}(0)+\tilde z \subseteq S_{(x^*,z^*)}(\bar x)+\tilde z = a(\bar x),$$
and $a \in  \mathcal{E}_{f^{cc'}}$. Hence, $\epi_K f^{cc'} \subseteq \epi_K f$ and $ f^{cc'}(x) \subseteq f_K(x) $, for all $x \in X$.
\qed
\end{proof}

\begin{theorem}
Let $f:X \to \mathcal{P}(Z)$ be a set-valued function. Then $f^{cc'} = K\text{-}\eco f$.
\end{theorem}
\begin{proof}
By Lemma \ref{lemma4}, $  f^{cc'} (x) \leq_K f(x)$, for all $x \in X$, hence $ f^{cc'}(x) \leq_K \Keco f(x)$, for all $x \in X$, since $f^{cc'}$ is a $K$-e-convex minorant of $f$.\\
Now we will show that $ \Keco f(x) \leq_K f^{cc'} (x)$, i.e. $ f^{cc'} (x) \subseteq \Keco f(x)$, for all $x \in X$.\\
Since $f(x)+K \subseteq \Keco f(x)$, for all $x \in X$, in particular, $\dom f \subseteq \dom \Keco f$. Take  $(x^*,  y^*, z^*, \alpha) \in X^* \times X^* \times (K^* \setminus\{0\}) \times \R$, such that $\dom f \subseteq H^-_{y^*, \alpha}$. Then, by Definition \ref{definition1},
\begin{equation}\label{eq5}
\sigma_f (x^*,  y^*, z^*,  \alpha)=\supp_{ \epi_K f}  \lbrace \langle x, x^* \rangle+\langle z, z^* \rangle \rbrace \leq \supp_{\epi_K \Keco f}  \lbrace \langle x, x^* \rangle+\langle z, z^* \rangle \rbrace
\end{equation}
Therefore, if $\dom f \subseteq H^-_{y^*, \alpha}$,
$$\sigma_f (x^*,  y^*, z^*,  \alpha) \leq \sigma_{\Keco f} (x^*,  y^*, z^*,  \alpha).$$
On the other hand, if  $\dom f \not \subseteq H^-_{y^*, \alpha}$, $\dom \Keco f \not \subseteq H^-_{y^*, \alpha}$ either, and 
hence, $\sigma_f (x^*,  y^*, z^*,  \alpha)= \sigma_{\Keco f} (x^*,  y^*, z^*,  \alpha)=+ \infty$. According to Lemma \ref{lemma4}, for all $x \in X$, $f^{cc'} (x) \subseteq {(\Keco f)}^{cc'} (x)=\Keco f(x)$, the last equality follows from Theorem \ref{theo5}.
\qed
\end{proof}

\begin{example}
Let us consider a well-known extended real-valued  function, the indicator function of a set $C \subseteq X$, $\delta_C: X \to \Ramp$,
\begin{equation*}
\delta_C (x)=\left\{
\begin{array}{ll}
0 & \text{if } x \in C, \\
+\infty & \text{otherwise.}
\end{array}
\right.
\end{equation*}
We will have that $C$ is e-convex iff $\delta_C$ is e-convex. On the other hand, denoting by $\delta_C^*$ the Fenchel conjugate of $\delta_C$,  we have that the c-conjugate  of $\delta$, $\delta^c_C: X^*\times X^* \times \R \to \Ramp$, is
\begin{equation*}
\delta^{c}_C (x^*, y^*, \alpha)=\left\{
\begin{array}{ll}
\delta^*_C(x^*) & \text{if }  C \subseteq H^{-}_{y^*,\alpha}, \\
+\infty & \text{otherwise,}
\end{array}
\right.
\end{equation*}
where $\delta_C^*(x^*)=\sigma_C (x^*)$ is the  well-known support function of $C$. Moreover, $\delta ^{cc'}_C=\delta_{\eco C}$.\\
Coming back to the set-valued functions framework, with $K$ a convex cone and taking into account that $\eco K + K= \eco K$, let us consider the set-valued  indicator function of a set $C \subset X$, $\Delta_C :X \to \mathcal{R}_K(Z)$,
\begin{equation*}
\Delta_C (x)=\left\{
\begin{array}{ll}
\eco K & \text{if } x \in C, \\
\emptyset & \text{otherwise,}
\end{array}
\right.
\end{equation*}
which is $K$-e-convex iff $C$ is e-convex. We calculate its c-conjugate (respect to $K$),
\begin{equation*}
\Delta^{c}_C (x^*, y^*, z^*, \alpha)=\left\{
\begin{array}{ll}
-\eco \left [ \bigcup_{ C} \lbrace \bar S_{(x^*,z^*)}(-x) \rbrace \right ] & \text{if }  C \subseteq H^{-}_{y^*,\alpha}, \\
\emptyset & \text{otherwise.}
\end{array}
\right.
\end{equation*}
Making a comparison with the (convex) conjugate set-valued function of $\Delta_C$ suggested in \cite[Section 4]{H2009}, $\Delta^*_C:X^* \times (K^*\setminus \{0\}) \to \mathcal{P}
(Z)$,
$$\Delta^*_C (x^*, z^*)=-\cl \left [ \bigcup_{ C} \lbrace \bar S_{(x^*,z^*)}(-x) \rbrace \right ], $$
we deduce that, if  $C \subseteq H^{-}_{y^*,\alpha}$, we could have $\Delta^{c}_C (x^*, y^*, z^*, \alpha) \subsetneq \Delta^*_C (x^*, z^*)$, this is not the case in the scalar framework, but recall that also $K$-convex hull concept is tighter in the set-valued context than  in the scalar one. Finally, we also have that $\Delta^{cc'}_C=\Delta_{\eco C}$.

\end{example}

\begin{acknowledgements}
The author sincerely thanks anonymous referees for their careful reading and thoughtful
comments. Their suggestions have significantly improved the quality of the paper.\\
Research partially supported by MINECO of Spain and ERDF of EU, Grant MTM2014-59179-C2-1-P.
\end{acknowledgements}

\section*{Conflict of interest}

The author declares that she has no conflict of interest.

\bibliographystyle{spmpsci}  
\bibliography{biblio}

\end{document}